\documentclass[11pt]{amsart} 

\usepackage[utf8]{inputenc} 

\usepackage{geometry} 
\usepackage{graphicx} 

\usepackage{booktabs} 
\usepackage{array} 
\usepackage{paralist} 
\usepackage{verbatim} 
\usepackage{subfig} 

\usepackage{amsfonts}
\usepackage{amssymb}
\usepackage{amsmath}
\usepackage{verbatim}
\usepackage{hyperref}
\usepackage{color}

\newtheorem{theorem}{Theorem}    
\newtheorem{proposition}[theorem]{Proposition}

\newtheorem{corollary}[theorem]{Corollary}
\newtheorem{lemma}[theorem]{Lemma}

\theoremstyle{definition}

\numberwithin{theorem}{section}
\numberwithin{definition}{section}
\numberwithin{equation}{section}

\usepackage{fancyhdr} 
\newcommand{\ra}{\rightarrow}

\def\C{\mathbb{C}}

\def\R{\mathbb{R}}

\newcommand{\ev}{\end{pmatrix}}

\def\bp{\mathbf{p}}
\def\bA{\mathbf{A}}
\def\bf{\mathbf{f}}
\def\bg{\mathbf{g}}
\def\bh{\mathbf{h}}
\def\bv{\mathbf{v}}
\def\T{\mathcal{T}}
\def\scriptO{\mathcal{O}}
\def\one{{\mathbf 1}}

\def\dist{\text{dist}}

\title{A Sharpened Inequality for Twisted Convolution}
\author{Kevin O'Neill}

\begin{document}
\maketitle

\begin{abstract}
Consider the trilinear form for twisted convolution on $\R^{2d}$:
\begin{equation*}
\T_t(\bf):=\iint f_1(x)f_2(y)f_3(x+y)e^{it\sigma(x,y)}dxdy,\end{equation*}
 where $\sigma$ is a symplectic form and $t$ is a real-valued parameter.
It is known that in the case $t\neq0$ the optimal constant for twisted convolution is the same as that for convolution, though no extremizers exist.
Expanding about the manifold of triples of maximizers and $t=0$ we prove a sharpened inequality for twisted convolution with an arbitrary antisymmetric form in place of $\sigma$.

\end{abstract} 


\section{Introduction}

Young's convolution inequality states that for dimensions $d\geq1$ and functions $f\in L^p(\R^d), g\in L^q(\R^d)$,

\begin{equation}\label{eq:Young's}
||f*g||_{L^r}\leq \bA^d_{\bp}||f||_{L^p}||g||_{L^q},
\end{equation}
where $p,q,r\in[1,\infty]$ with $\frac{1}{p}+\frac{1}{q}=1+\frac{1}{r}$. $\bA^d_{\bp}=\prod_{j=1}^3 C_{p_j}^d$ is the optimal constant, where $C_p=p^{1/p}/p'^{1/p'}$, and $p'$ is the conjugate exponent of $p$ \cite{MR0385456}, \cite{MR0412366}. For the purpose of this paper, it is convenient to use the following, related trilinear form:

\begin{equation}
\T(f_1,f_2,f_3)=\iint f_1(x)f_2(y)f_3(x+y)dxdy.
\end{equation}

Through duality, one may rewrite \eqref{eq:Young's} as

\begin{equation}\label{eq:Young's trilinear}
\left|\T(\bf)\right| \leq \bA^d_{\bp}\prod_{j=1}^3||f_j||_{L^{p_j}}
\end{equation}
for all $\bf=(f_j\in L^{p_j}(\R^d):j=1,2,3)$, with $\sum p_j^{-1}=2$ and $\bp=(p_j:j=1,2,3)\in[1,\infty]^3$.

From here on out, we take $p_j\in(1,\infty)$. In \cite{MR0412366}, Brascamp and Lieb show that the maximizers of \eqref{eq:Young's trilinear} are precisely the triple of Gaussians $\bg =(e^{-\pi p_j'|x|^2}: j=1,2,3)$ and its orbit under the following symmetries.


\begin{itemize}
\item $(f_1,f_2,f_3)\mapsto(af_1,bf_2,cf_3)$ for $a,b,c\neq0$. (Scaling)
\item $(f_1,f_2,f_3)\mapsto(M_{\xi}f_1,M_{\xi}f_2,M_{-\xi}f_3)$, where $M_{\xi}f(x)=e^{ix\cdot\xi}$ for $\xi\in\R^d$. (Modulation)
\item $(f_1,f_2,f_3)\mapsto(\tau_{v_1}f_1,\tau_{v_2}f_2,\tau_{v_1+v_2}f_3)$, where $\tau_vf(x)=f(x+v)$ for $v\in\R^d$. (Translation)
\item $(f_1,f_2,f_3)\mapsto(f_1\circ\psi,f_2\circ\psi,f_3\circ\psi)$, where $\psi$ is an invertible linear map on $\R^d$. (Diagonal Action of the General Linear Group)
\end{itemize}

Note that these symmetries do not necessarily preserve $|\T(\bf)|$, but they do preserve $|\Phi(\bf)|$, where $\Phi(\bf):=\frac{\T(\bf)}{\prod_j||f_j||_{p_j}}$.

Let $\scriptO_C(\bf)$ denote the orbit of the triple $\bf$ under the above symmetries. Define the distance from $\bg$ to $\scriptO_C(\bf)$ as

\begin{equation}\label{eq:define distance}
\dist_\bp(\scriptO_C(\bf),\bg):=\inf_{\bh\in\scriptO_C(\bf)}\max_j||h_j-g_j||_{p_j}.
\end{equation}

Note that the symmetries of an operator preserve the (normalized) distance of a triple from the manifold of maximizers.

Christ \cite{ChristSY} proved the following quantitative stability theorem for Young's convolution inequality.

\begin{theorem}\label{thm:ChristSY}
Let $K$ be a compact subset of $(1,2)^3$. Let $\bp$ satisfy $\sum_{j=1}^3p_j^{-1}=2$. For each $d\geq1$, there exists $c>0$ such that for all $\bp\in K$ and all $\bf\in L^\bp(\R^d)$,
\begin{equation}
|\T(\bf)|\leq \left(\bA^d_\bp-c\dist_\bp(\scriptO_C(\bf),\bg)^2\right)\prod_j||f_j||_{p_j}.
\end{equation}
\end{theorem}

One may instead state the above theorem in terms of the distance of a triple $\bf$ from the set of all triples of maximizers (that is, $\scriptO_C(\bg)$), as is done in \cite{ChristSY}. However, the distance defined in \eqref{eq:define distance} is more useful for analogy with our current analysis. 

It is also shown that the conclusion of Theorem \ref{thm:ChristSY} is true for $\bp\in(1,2]^3$ provided one does not require the same $c$ for all $\bp$ in a region. (However it is not known if this uniformity fails.) Furthermore, the conclusion in this particular quantitative form is false whenever any $p_j=1$ or $p_j>2$.

The purpose of this paper is to prove a similar quantitative stability result for twisted convolution. Let $t\geq0$ be a parameter and let $f_j\in L^{p_j}(\R^{2d})$, where $\R^{2d}$ is viewed as $\R^d\times\R^d=\{(x',x'''):x',x''\in\R^d\}$. Define the trilinear twisted convolution form with parameter $t$ as

\begin{equation}\label{eq:deftform}
\T_t(\bf):=\iint f_1(x)f_2(y)f_3(x+y)e^{it\sigma(x,y)}dxdy,
\end{equation}
where $\sigma(x,y)=x'\cdot y''-x''\cdot y'$ is the symplectic form. It is often useful to write $\sigma(x,y)=x^tJy$, where $J$ is the matrix
\begin{equation}
J=\begin{pmatrix}
0&I_d\\
-I_d&0
\end{pmatrix},
\end{equation}
and $I_d$ is the $d\times d$ identity matrix.

When $t=0$, \eqref{eq:deftform} becomes the trilinear form representing convolution. When $t\neq0$, it is obvious through the inequality 
\begin{equation}\label{eq:abs values}
|\T_t(\bf)|\leq\T(|\bf|)
\end{equation}
that $\T_t$ is bounded for any triple $\bp$ of exponents for which $\T$ is bounded. It is also known that for $t\neq0$, $\T(\bf,t)$ is also bounded for $(p_1,p_2,p_3)=(2,2,2)$ and the full range of exponents implied by interpolation (see Chapter XII.4 of \cite{SteinHA}, for instance). However, the particular conclusion we desire is false in the case $\sum_jp_j^{-1}\neq2$ since $\T_0=\T$ is unbounded. 


By \eqref{eq:abs values}, it is easy to see that $\T_t$ has norm at most $A_\bp^{2d}$, the optimal constant for Young's convolution inequality. Furthermore, the optimal constant may be seen to equal $A_\bp^{2d}$ by taking a triple of Gaussians which optimize Young's inequality and dilating them to concentrate at the origin so the oscillation of the twisting factor has negligible effect. However, no extremizers of $\T_t$ exist for fixed $t\neq0$. \cite{KleinRusso}


One challenge to dealing with the above form directly arises because the symmetry group of $\T$ contains the general linear group $Gl(2d)$, while $\T_t$ does not; the only linear transformations which preserve $\sigma$ are the symplectomorphisms. To avoid this issue, it helps to introduce the following trilinear form:

\begin{equation}\label{eq:defLform}
\T_A(\bf):=\iint f_1(x)f_2(y)f_3(x+y)e^{it\sigma(Ax,Ay)}dxdy,
\end{equation}
where $A:\R^{2d}\ra\R^{2d}$ is an arbitrary linear map. Replacing $x$ with $Lx$ and $y$ with $Ly$ for an invertible matrix $L$ sends $A$ to $A\circ L$, and the functional remains of the form \eqref{eq:defLform}. Boundedness properties of $\T_A$ follow directly from those of $\T_t$ and a change of coordinates.

The symmetries of $\T_A$ are similar to the those of $\T$ with some slight modifications, though they reduce to the symmetries of $\T(\bf)$ when $A=0$. Here, the symmetries preserve $|\Phi(\bf,A)|$, where $\Phi(\bf,A)=\frac{\T(\bf,A)}{\prod_j||f_j||_{p_j}}$.

\begin{itemize}
\item $(f_1,f_2,f_3,A)\mapsto(af_1,bf_2,cf_3,A)$, where $a,b,c\in\C$. (Scaling)
\item $(f_1,f_2,f_3,A)\mapsto(M_{\xi}f_1,M_{\xi}f_2,M_{-\xi}f_3,A)$. (Modulation)
\item $(f_1,f_2,f_3,A)\mapsto(M_{A^TJAv_2}\tau_{v_1}f_1,M_{-A^TJAv_1}\tau_{v_2}f_2,\tau_{v_1+v_2}f_3,A)$, where $A^T$ represents the transpose of the matrix $A$. (Translation/ Modulation Mix)
\item $(f_1,f_2,f_3,A)\mapsto\T(f_1\circ\psi,f_2\circ\psi,f_3\circ\psi,A\circ\psi)$, where $\psi\in Gl(d)$. (Diagonal Action of the General Linear Group)
\end{itemize}
Note that only the last of these symmetries alters $A$.

Let $\scriptO_{TC}(\bf,A)$ denote the orbit of $(\bf,A)$ under the above symmetries.

Now, it is less obvious how to represent the distance of $A$ from the zero transformation than it was when our parameter was just a real number $t$. One may naively suggest that $||A||$ will play a role, but this approach ignores the role of the symplectic group. The real symplectic group $Sp(2d)$ is defined as the set of invertible $(2d)\times(2d)$ matrices $S$ such that $S^TJS=J$. Equivalently, $Sp(2d)$ may be viewed as the set of coordinate changes which preserve $\sigma$. Under this view, we see that $\sigma(Ax,Ay)=\sigma(SAx,SAy)$ for any $S\in Sp(2d)$. Thus, replacing $A$ with $S\circ A$ should not change our distance.

With this in mind, define the distance from $\scriptO_{TC}(\bf,A)$ to $(\bg,0)$ by

\begin{equation}\label{eq:define TC distance}
\dist_\bp(\scriptO_{TC}(\bf,A),(\bg,0))^2:=\inf_{(\bh,M)\in\scriptO_{TC}(\bf,A)}\left[\max_j||h_j-g_j||_{p_j}^2+||M^TJM||^2\right]
\end{equation}

A useful fact in analyzing this distance is that $\inf_{S\in Sp(2d)}||S\circ A||^2=||A^TJA||$. (See Lemma 10.1 of \cite{ChristHeisenberg}.)  Define $||\bf||_\bp=\max_j||f_j||_{p_j}$. We now state our main theorem.

\begin{theorem}\label{theorem:main}
Let $K$ be a compact subset of $(1,2)^3$. For each $d\geq1$, there exists $c>0$ such that for all $\bp\in K$ with $\sum_{j=1}^3p_j^{-1}=2$, $\bf\in L^\bp(\R^{2d})$, and $(2d)\times(2d)$ matrices $A$,
\begin{equation}
|\T_A(\bf)|\leq \left(\bA^{2d}_\bp-c\dist_\bp(\scriptO_{TC}(\bf,A),(\bg,0))^2\right)\prod_j||f_j||_{p_j}.
\end{equation}
\end{theorem}
By setting $A=t^{1/2}I_{2d}$ in Theorem \ref{theorem:main} (where $I_{2d}$ is the $(2d)\times(2d)$ identity matrix), one obtains the following corollary. However, one is cautioned that the orbit in this expression refers to the symmetries of $\T_A$, not those of $\T_t$.

\begin{corollary}\label{corollary:only}
Let $K$ be a compact subset of $(1,2)^3$. For each $d\geq1$, there exists $c>0$ such that for all $\bp\in K$ with $\sum_{j=1}^3p_j^{-1}=2$, $\bf\in L^\bp(\R^{2d})$, and $|t|\leq1$,
\begin{equation}
|\T_t(\bf)|\leq \left(\bA^{2d}_\bp-c\dist_\bp(\scriptO_{TC}(\bf,t^{1/2}I_{2d}),(\bg,0))^2\right)\prod_j||f_j||_{p_j}.
\end{equation}
\end{corollary}

The reason one uses $t^{1/2}I_{2d}$ rather than $tI_{2d}$ is so the $||M^TJM||^2$ term appearing in \eqref{eq:define TC distance} is proportional to $t^2$, rather than $t^4$. An alternative form of Corollary \ref{corollary:only} states the function $\epsilon(\delta)$ in Theorem \ref{thm:perturbative1} may be taken to be $C\sqrt{\delta}$ for some $C>0$.


The methods in this paper follow the general approach found in \cite{ChristSY} and \cite{BE} in which one takes a Taylor-like expansion of the given operator and diagonalizes the resulting quadratic form.

We will often use $C$ or $c$ to denote an arbitrary constant in $(0,\infty)$ which may change from line to line but always be independent of functions found in the equation.


\section{Reduction to Perturbative Case}

Our argument centers around an expansion of $\T(\bf,A)$ which requires a reduction to small perturbations. To this end, the following result from \cite{ChristHeisenberg} is essential.


\begin{theorem}\label{thm:perturbative1}
Let $d\geq 1$. Let $K$ be a compact subset of $(1,2)^3$ for which each $\bp\in K$ satisfies $\sum_{j=1}^3p_j^{-1}=2$. Then, there exists a function $\delta\mapsto\epsilon(\delta)$ (depending only on $K$ and $d$) satisfying $\lim_{\delta\ra0}\epsilon(\delta)=0$ with the following property. Let $\bf\in L^\bp(\R^{2d})$ and suppose that $||f_j||_{p_j}\neq0$ for each $1\leq j\leq3$. Let $\delta\in(0,1)$ and suppose that $|\T(\bf,t)|\geq(1-\delta)\bA_\bp^{2d}\prod_j||f_j||_{p_j}$. Then there exist $S\in Sp(2d)$ and a triple of Gaussians $\mathbf{G}=(G_1,G_2,G_3)$ such that $G_j^\natural=G_j\circ S$ satisfy
\begin{equation}
||f_j-G_j^\natural||_{p_j}<\epsilon(\delta)||f_j||_{p_j}
\end{equation}
for $1\leq j\leq3$ and
\begin{equation}
G_j(x)=c_je^{\pi p_j'|L(x-a_j)|^2}e^{ix\cdot v}e^{it\sigma(\tilde{a_j},x)}
\end{equation}
where $v\in\R^{2d}, 0\neq c_j\in\C, a_1+a_2+a_3=0, \tilde{a_3}=0,\tilde{a_1}=a_2,\tilde{a_2}=a_1, L\in Gl(2d)$, and
\begin{equation}
|t|\cdot||L^{-1}||^2\leq\epsilon(\delta).
\end{equation}
\end{theorem}

Here is a rephrasing of Theorem \ref{thm:perturbative1}.

\begin{theorem}\label{thm:perturbative2}
Let $d\geq 1$. Let $K$ be a compact subset of $(1,2)^3$ for which each $\bp\in K$ satisfies $\sum_{j=1}^3p_j^{-1}=2$. Then, there exists a function $\delta\mapsto\epsilon(\delta)$ (depending only on $K$ and $d$) satisfying $\lim_{\delta\ra0}\epsilon(\delta)=0$ with the following property. Let $\bf\in L^\bp(\R^{2d})$ and suppose that $||f_j||_{p_j}\neq0$ for each $1\leq j\leq3$. Let $\delta\in(0,1)$ and suppose that $|\T_A(\bf)|\geq(1-\delta)\bA_\bp^{2d}\prod_j||f_j||_{p_j}$. Then,
\begin{equation}\label{eq:pert2conc}
\dist_\bp(\scriptO_{TC}(\bf,A),(\bg,0))<\epsilon(\delta)
\end{equation}
\end{theorem}

\begin{proof}[Proof of Theorem \ref{thm:perturbative1} $\Rightarrow$ Theorem \ref{thm:perturbative2}]
By a standard approximation argument, it suffices to prove Theorem \ref{thm:perturbative2} for invertible maps $A$, as each noninvertible map is arbitrarily close to an invertible map.


Suppose that $|\T(\bf,A)|\geq(1-\delta)\bA_\bp^{2d}\prod_j||f_j||_{p_j}$. Then invoking the symmetry of diagonal action of the general linear group, 
\begin{equation}\label{eq:filler label}
|\T(\bf\circ A^{-1},I_{2d})|\geq(1-\delta)\bA_\bp^{2d}\prod_j||f_j\circ A^{-1}||_{p_j},
\end{equation}
where $\bf\circ A^{-1}=(f_j\circ A^{-1}:j=1,2,3)$.

Applying Theorem \ref{thm:perturbative1} under the case $t=1$, there exists $S_0\in Sp(2d)$ and a triple of Gaussians $\mathbf{G}=(G_1,G_2,G_3)$ such that
\begin{equation}\label{eq:norm control}
||f_j\circ A^{-1}-G_j\circ S_0||_{p_j}<\epsilon(\delta)||f_j\circ A^{-1}||_{p_j}
\end{equation}
for $1\leq j\leq3$ and
\begin{equation}
G_j(x)=c_je^{\pi p_j'|L(x-a_j)|^2}e^{ix\cdot v}e^{it\sigma(\tilde{a_j},x)}
\end{equation}
where $v\in\R^{2d}, 0\neq c_j\in\C, a_1+a_2+a_3=0, \tilde{a_3}=0,\tilde{a_1}=a_2,\tilde{a_2}=a_1, L\in Gl(2d)$, and
\begin{equation}
||L^{-1}||^2\leq\epsilon(\delta).
\end{equation}

By a combination of translations, modulations, scalings, and compositions with invertible linear maps, \eqref{eq:norm control} becomes
\begin{equation}
||h_j-g_j||_{p_j}<\epsilon(\delta)||h_j||_{p_j},
\end{equation}
where $h_j$ is $f_j\circ A^{-1}$ composed with said operations.

Since $\mathbf{G}$ was the composition of $\bg$ with the stated symmetries of $\T_A$, we see that $\bh$ is obtained by the composition of $\bf\circ A^{-1}$ with symmetries of $\T_A$ by the following reasoning. Three of these symmetries (scaling, modulation, and the diagonal action of the general linear group) may trivially be inverted by symmetries of the same form. To address the inversion of the translation/modulation mix, one observes that $\tau_{w_j}M_{B^TJB\tilde{w_j}}f=e^{iB^TJB\tilde{w_j}\cdot w_j}M_{B^TJB\tilde{w_j}}\tau_{w_j}f$ for matrices $B$ and vectors $w_j$. Hence, $\bh$ is obtained from $\bf\circ A^{-1}$ through the inverses of the symmetries applied initially to $\bg$ to obtain $\mathbf{G}$ but with an additional scaling symmetry.

The only above symmetry which changes the matrix $B$ in $\T_B$ is the diagonal action of the general linear group. Following the use of this symmetry above, one obtains from \eqref{eq:filler label} that $(\bh,S_0^{-1}\circ L^{-1})\in\scriptO_{TC}(\bf,A)$.

We now see that
\begin{align*}
\dist_\bp(\scriptO_{TC}(\bf,A),(\bg,0))^2&\leq\max_j||h_j-g_j||_{p_j}^2+\inf_{S\in Sp(2d)}||S\circ  S_0^{-1}\circ L^{-1}||^4\\
&\leq \epsilon(\delta)^2+||S_0S_0^{-1}\circ L^{-1}||^4\\
&\leq \epsilon(\delta)^2+||L^{-1}||^4\leq 2\epsilon(\delta)^2.
\end{align*}


%
\end{proof}


As a corollary to Theorem \ref{thm:perturbative2}, it suffices to prove Theorem \ref{theorem:main} in the case in which $\dist_\bp(\scriptO_{TC}(\bf,A),(\bg,0))<\delta_0$ for some $\delta_0>0$. Theorem \ref{thm:perturbative2} guarantees that there are no sequences of $(\bf_n,A_n)$ at distance greater than $\delta_0$ such that $\T_{A_n}(\bf_n)/(\prod_j||f_{n,j}||_{p_j})$ converges to $A_\bp^{2d}$. Thus, for $(\bf,A)$ at distance at least $\delta_0$, $\T_A(\bf)$ must have a maximum strictly less than $A_\bp^{2d}$. While $||A^TJA||\ra\infty$ for an appropriate sequence of matrices $A$, $\dist_\bp(\scriptO_{TC}(\bf,A),(\bg,0))$ remains bounded above as th symmetries of $\T_A$ ensure there exists $(\bh,M)\in\scriptO_{TC}(\bf,A)$ with $||M^TJM||\leq1$. Therefore, the conclusion of Theorem \ref{theorem:main} holds for distances greater than $\delta_0$.

\section{Treating Some Terms of the Expansion}

In this section, we consider $\T_A(\bg+\bf)$, where $A$ is a $(2d)\times(2d)$ matrix, $\bg =(g_j=e^{-\pi p_j'|x|^2}: j=1,2,3)$ and $\bf\in L^{\bp}(\R^{2d})$ are small perturbations. (This change in notation of $\bf$ from functions close to $\bg$ to the differences will continue for the remainder of the paper.) As in \cite{ChristSY}, we may assume $\int g_j^{p_j-1}f_j=0$ via the scaling symmetry.

In short, we will expand $\T(\bg+\bf,A)=\T_0(\bg+\bf)+(\T_A-\T_0)(\bg+\bf)$ and use the multilinearity of $\T_0$ and $\T_A$ to get sixteen terms of eight different types. Before writing out the expansion, we prove a few lemmas about its terms and describe a useful decomposition.

Following \cite{ChristHY} and \cite{ChristSY}, let $\eta>0$ be a small parameter to be chosen later (see Proposition \ref{prop:only}). For each $1\leq j\leq3$, decompose $f_j=f_{j,\sharp}+f_{j,\flat}$, where
\begin{equation}\label{eq:def fsharp}
f_{j,\sharp}=\left\{
     \begin{array}{lr}
       f_j(x) \text{ if }|f_j(x)|\leq \eta g_j(x)\\
       0 \text{ otherwise, }
     \end{array}
   \right.
\end{equation}
and $f_{j,\flat}=f_j-f_{j,\sharp}$. The purpose of this decomposition is twofold. First, it is used in the analysis of \cite{ChristSY} to analyze the quadratic form in the expansion with $L^2$ functions. Using the same decomposition allows us to borrow from that analysis in Proposition \ref{prop:only}, a version of Theorem \ref{thm:ChristSY} with an additional favorable term. Second, the decomposition is used to reduce to the case of $f_j=f_{j,\sharp}$, which concentrates closer to the origin, allowing for control of the third order term in Lemma \ref{lemma:3rd order}.

\begin{lemma}\label{lemma:trivial}
$(\T_A-\T_0)(\bf)=O(||\bf||_\bp^3).$
\end{lemma}
\begin{proof}
This claim follows trivially from the uniform boundedness of $\T_A$ and $\T_0$.
\end{proof}

The following lemma represents our main use of the $f_j=f_{j,\sharp}+f_{j,\flat}$ decomposition and the swapping of $f_j$ for $f_{j,\sharp}$ will be justified later.

\begin{lemma}\label{lemma:3rd order}
$(\T_A-\T_0)(f_{1,\sharp},f_{2,\sharp},g_3)=o(||\bf||_\bp^2+||A^TJA||^2)$ with decay rate depending only on $\eta$.
\end{lemma}

Lemma \ref{lemma:3rd order} also applies to the other two terms of this type.

Note that the trivial bound
\begin{equation}\label{eq:trivial bound}
|(\T_A-\T_0)(h_1,h_2,g_3)|=O\left(||h_1||_{p_1}||h_2||_{p_2}\right)
\end{equation}
is insufficient to deal with the above term directly since it provides a second order control of a term which should heuristically be third order. However, \eqref{eq:trivial bound} still plays a useful role in the proof of Lemma \ref{lemma:3rd order}.

\begin{proof}
First, suppose that $||A^TJA||^3\geq||f_{1,\sharp}||_{p_1}||f_{2,\sharp}||_{p_2}$. Note that by our reduction to small perturbations in Theorem \ref{thm:perturbative2}, $||A^TJA||$ may be taken small enough that $||A^TJA||^3\leq||A^TJA||^2$. By \eqref{eq:trivial bound},
\begin{equation}
(\T_A-\T_0)(f_{1,\sharp},f_{2,\sharp},g_3)\leq C||f_{1,\sharp}||_{p_1}||f_{2,\sharp}||_{p_2}\leq||A^TJA||^3=o(||\bf||_\bp^2+||A^TJA||^2)
\end{equation}
and we are done.

So suppose that $||A^TJA||^3<||f_{1,\sharp}||_{p_1}||f_{2,\sharp}||_{p_2}$. Now, for $j=1,2$, write $f_{j,\sharp}=f_{j,\sharp,\leq M_j}+f_{j,\sharp,>M_j}$, where $f_{j,\sharp,\leq M_j}=f_{j,\sharp}\one_{B(0,M_j)}$ and $f_{j,\sharp,>M_j}=f_{j,\sharp}\one_{B(0,M_j)^c}$. In the above, $\one_E$ refers to the indicator function of the set $E$, $B(x_0,R)$ refers to the closed ball of radius $R$ centered at $x_0$, $E^c$ is the complement of the set $E$, and $M_j$ is chosen so that
\begin{equation}\label{eq:smallnorm}
||f_{j,\sharp,>M_j}||_{p_j}=||f_{j,\sharp}||_{p_j}^2.
\end{equation}

Note that $M_j$ is dependent on $\eta$.

We claim that $M_j\leq C\log(||f_{j,\sharp}||_{p_j}^{-1})$. To see this, observe that for given $\eta$ and $||f_{j,\sharp}||_{p_j}$ and varying $f_{j,\sharp}$, $M_j$ is maximized when $f_{j,\sharp}=\eta g_j$ on $B(0,M)^c$ and $f_{j,\sharp}=0$ on $B(0,M)$, where $M$ is the positive real number that leads to the appropriate value of $||f_{j,\sharp}||_{p_j}$. (Here, $M<M_j$ since $||f_{j,\sharp}||_{p_j}$ is small.) It suffices to find an upper bound for $M_j$ in this scenario. We integrate with respect to spherical coordinates to obtain
\begin{align*}
||f_{j,\sharp}||_{p_j}^2&=||f_{j,\sharp,>M_j}||_{p_j}\\
&=\int_{S^{d-1}}\left[\int_{M_j}^\infty \eta e^{-\pi p_j'r^2}r^{2d-1}dr\right]d\sigma(\theta)\\
&=C_d\eta\int_{M_j}^\infty \eta e^{-\pi p_j'r^2}r^{2d-1}dr\\
&=O(M_j^{2d-2}e^{-\pi p_j'M_j^2}).
\end{align*}
Thus, $||f_{j,\sharp}||_{p_j}\leq Ce^{-M_j}$, proving our claim.


Expand 
\begin{multline*}
(\T_A-\T_0)(f_{1,\sharp},f_{2,\sharp},g_3)=(\T_A-\T_0)(f_{1,\sharp,>M_1},f_{2,\sharp,>M_2},g_3)+(\T_A-\T_0)(f_{1,\sharp,>M_1},f_{2,\sharp,\leq M_2},g_3)\\
+(\T_A-\T_0)(f_{1,\sharp,\leq M_1},f_{2,\sharp,>M_2},g_3)+(\T_A-\T_0)(f_{1,\sharp,\leq M_1},f_{2,\sharp,\leq M_2},g_3)
\end{multline*}
The first three of these terms may be treated by combining the trivial bound \eqref{eq:trivial bound} with \eqref{eq:smallnorm}.

Let $R=B(0,M_1)\times B(0,M_2)\subset \R^{2d}\times\R^{2d}$. The absolute value of the remaining term is
\begin{align*}
|(\T_A-\T_0)(f_{1,\sharp},f_{2,\sharp},g_3)|&\leq\iint_R|f_{1,\sharp}(x)|\cdot|f_{2,\sharp}(y)|\cdot g_3(x+y)\cdot|\sigma(Ax,Ay)|dxdy\\
&\leq C||f_{1,\sharp}||_{p_1}||f_{2,\sharp}||_{p_2}||g_3||_{p_3}||A^TJA||M_1M_2\\
&\leq C||f_{1,\sharp}||_{p_1}^{4/3}||f_{2,\sharp}||_{p_2}^{4/3}\log(||f_{1,\sharp}||_{p_1}^{-1})\log(||f_{2,\sharp}||_{p_2}^{-1})=o(||\bf||_\bp^2)
\end{align*}

\end{proof}

\begin{lemma}\label{lemma:onesigma}
For all $f\in L^{p_1}(\R^{2d})$

\begin{equation}
\iint f(x)g_2(y)g_3(x+y)\sigma(Ax,Ay)dxdy=0
\end{equation}
\end{lemma}
The conclusion also applies to the same integral with $(g_1,f,g_3)$ or $(g_1,g_2,f)$ in place of $(f,g_2,g_3)$ (with $f\in L^{p_j}$ for the appropriate $j\in\{1,2,3\}$).
\begin{proof}
Since $\sigma(Ax,Ay)=x^TA^TJAy$ is an antisymmetric bilinear form, we may diagonalize $A^TJA$ as $Q^T\Sigma Q$ for some orthogonal $Q$ and
\begin{equation}
\Sigma=\begin{pmatrix}
0&a_1&...&0&0\\
-a_1&0&...&0&0\\
\vdots&\vdots&\ddots&\vdots&\vdots\\
0&0&...&0&a_d\\
0&0&...&-a_d&0
\end{pmatrix},
\end{equation}
where $a_k\in\R$ and $\pm a_ki$ are the eigenvalues of $A^TJA$. Since $g_j(x)=e^{-\pi p_j'|x|^2}$, $g_2$ and $g_3$ remain unchanged under an orthogonal change of coordinates. Thus, the above is equal to

 and we may write the above as
\begin{equation}
\iint f(Qx)g_2(y)g_3(x+y)\sum_{k=1}^d a_k(x_{2k-1}y_{2k}-x_{2k}y_{2k-1})dxdy.
\end{equation}
Since $f(x)$ is an arbitrary function of $x$, $f(Qx)$ is also an arbitrary function of $x$, so it suffices to show that
\begin{equation}
\int g_2(y)g_3(x+y)\sum_{k=1}^d a_k(x_{2k-1}y_{2k}-x_{2k}y_{2k-1})dy=0
\end{equation}
for all $x\in\R^{2d}$.

By linearity and permutation of coordinates, it suffices to show that
\begin{equation}
\int g_2(y)g_3(x+y)(x_1y_2-x_2y_1)dy=0.
\end{equation}
Writing $e^{-\pi p_j'|w|^2}=e^{-\pi p_j'(w_1^2+w_2^2)}e^{-\pi p_j'(w_3^2+...+w_{2d}^2)}$, the above integral factors into
\begin{equation}
\int g_2(y_1,y_2)g_3(x_1+y_1,x_2+y_2)(x_1y_2-x_2y_1)dy_1dy_2\cdot\int g_2(\tilde{y})g_3(\tilde{x}+\tilde{y})d\tilde{y},
\end{equation}
where $x=(x_1,x_2,\tilde{x}),y=(y_1,y_2,\tilde{y})$, and through abuse of notation, $g_j(w)=e^{-p_j'|w|^2}$ for $w$ in any dimension. It now suffices to show the first factor is zero.

Expanding this factor gives
\begin{multline}
x_1\int y_2g_2(y_2)g_3(x_2+y_2)dy_2\cdot\int g_2(y_1)g_3(x_1+y_1)dy_1\\-x_2\int y_1g_2(y_1)g_3(x_1+y_1)dy_1\cdot\int g_2(y_2)g_3(x_2+y_2)dy_2.
\end{multline}

An elementary computation shows that $g_2*g_3=Cg_1$ and $\int yg_2(y)g_3(x+y)dy=C'xg_1(x)$, hence the above becomes
\begin{equation}
x_1\cdot C'x_2g_1(x_2)\cdot Cg_1(x_1)-x_2\cdot C'x_1g_1(x_1)\cdot Cg_1(x_2)=0.
\end{equation}
\end{proof}

If $S$ is a list of parameters, let $A\approx_S B$ mean there exists a $C>0$ depending only on elements of $S$ such that $A\leq CB$ and $B\leq CA$.

\begin{lemma}\label{lemma:twosigma}
For $\bg$ and $A$ as above,
\begin{equation}
\iint g_1(x)g_2(y)g_3(x+y)\sigma^2(Ax,Ay)dxdy\approx_{d,\bp}||A^TJA||^2.
\end{equation}
\end{lemma}

\begin{proof}
As in the proof of Lemma \ref{lemma:onesigma}, one may use an orthogonal change of coordinates to reduce to the computation of
\begin{equation}
\iint g_1(x)g_2(y)g_3(x+y)\left[\sum_{k=1}^da_k(x_{2k-1}y_{2k}-x_{2k}y_{2k-1})\right]^2dxdy.
\end{equation}
Expanding the square gives
\begin{equation}
\sum_{j,k=1}^da_ja_k\iint g_1(x)g_2(y)g_3(x+y)(x_{2j-1}y_{2j}-x_{2j}y_{2j-1})(x_{2k-1}y_{2k}-x_{2k}y_{2k-1})dxdy.
\end{equation}
By factoring the $g_j$ and computing the above integrals two coordinates at a time as in the proof of Lemma \ref{lemma:onesigma}, one finds that the cross terms are zero. Thus, the original integral is equal to a function to $d$ and $\bp$ alone times $\sum_{k=1}^da_k^2$. Recall that $\pm a_ki$ are the eigenvalues of $A^TJA$, so $||A^TJA||^2=\max_k|a_k|^2$ and the two expressions are equivalent.
\end{proof}

At this point, it is tempting to expand $\T_A(\bg+\bf)$, using the previous four lemmas to treat the $(\T_A-\T_0)$ terms (to get $-c||A^TJA||^2$) and Theorem \ref{thm:ChristSY} to treat the $\T_0$ terms (and get $A_\bp^{2d}-c||\bf||_\bp^2$). However, Theorem \ref{thm:ChristSY} may only be applied directly when the perturbative terms $f_j$ represent the projective distance from the orbit of the original functions to $\bg$. The subtle difference here is that the $f_j$ which represent the minimum value of $||\bf||_\bp^2$ may not be the same functions which represent the minimum value of $||\bf||^2+||A^TJA||^2$.

For this reason, we will delve somewhat into the proof of Theorem \ref{thm:ChristSY} and show that it is possible to obtain the same circumstances which lead to a $-c||\bf||^2$ decay.

\section{Balancing Lemma}

For $t>0$ and $n=0,1,2,...,$ let $P_n^{(t)}$ denote the real-valued polynomial of degree $n$ with positive leading coefficient and $||P_n^{(t)}e^{-t\pi x^2}||_{L^2(\R)}=1$ which is orthogonal to $P_k^{(t)}e^{-t\pi x^2}$ for all $0\leq k<n$.

For $d>1$, $\alpha=(\alpha_1,...,\alpha_d)\in\{0,1,2,...\}^d$, and $x=(x_1,...,x_d)\in\R^d$, define
\begin{equation}
P_\alpha^{(t)}(x)=\prod_{k=1}^dP_{\alpha_k}^{(t)}(x_k).
\end{equation}

Let $\tau_j=\frac{1}{2}p_jp_j'$ ($j=1,2,3$). In \cite{ChristSY}, the following is proved en route to the main theorem.

\begin{proposition}\label{prop:only}
Let $\delta_0>0$ be sufficiently small. There exists $c,\tilde{c}>0$ and a choice of $\eta>0$ in the $f_j=f_{j,\sharp}+f_{j,\flat}$ decomposition such that the following holds. Suppose $||\bf||_\bp<\delta_0$ and $f_j$ satisfy the following orthogonality conditions:
\begin{itemize}
\item $\langle Re(f_j),P_\alpha^{(\tau_j)}g_j^{p_j-1}\rangle=0$ whenever $\alpha=0$, $|\alpha|=1$ and $j\in\{1,2\}$, or $|\alpha|=2$ and $j=3$.
\item $\langle Im(f_j),P_\alpha^{(\tau_j)}g_j^{p_j-1}\rangle=0$ whenever $\alpha=0$ or $|\alpha|=1$ and $j=3$.
\end{itemize}
Then,
\begin{equation}\label{eq:Christ conclusion}
\frac{\T_0(\bg+\bf)}{\prod_j||g_j+f_j||_{p_j}}\leq A_\bp^{2d}-c||\bf||_\bp^2-\tilde{c}\sum_j||f_{j,\flat}||_{p_j}^{p_j}.
\end{equation}
\end{proposition}

The above proposition is not stated as an explicit result of \cite{ChristSY}. However, \eqref{eq:Christ conclusion} is, in effect, the penultimate line of the proof of Theorem \ref{thm:ChristSY} in Section 8 of \cite{ChristSY}. (The one difference is that $c||\bf||_\bp^2$ is replaced by $\sum_j||f_{j,\sharp}g_j^{(p_j-2)/2}||_2^2$ in the line in \cite{ChristSY}, though it is shown the latter majorizes a constant multiple of the former.)

We cite this particular intermediate result in order to take advantage of the $f_j=f_{j,\sharp}+f_{j,\flat}$ decomposition. The terms in Lemma \ref{lemma:3rd order} involve $f_{j,\sharp}$ in place of $f_j$ so \eqref{eq:Christ conclusion} is used to deal with the case that $f_{j,\flat}$ makes up a significant portion of the $L^{p_j}$ norm of $f_j$.


The goal of this section is to reduce to the situation in which the hypotheses of Proposition \ref{prop:only} apply. This is done through the use of the following balancing lemma.

\begin{lemma}[Balancing Lemma]
Let $d\geq1$ and $\bp\in(1,2]^3$ with $\sum_jp_j^{-1}=2$. There exists $\delta_0>0$ such that if $||F_j-g_j||_{p_j}\leq\delta_0$, $||A^TJA||\leq\delta_0$, and $\langle F_j-g_j,g_j^{p_j-1}\rangle=0$, then there exist $v_j\in \R^{2d}$ satisfying $v_1+v_2+v_3=0$, $a_j\in\C$, $\xi\in\R^{2d}$, and a $(2d)\times(2d)$ matrix $\psi$ such that
\begin{equation}
\sum_j(|v_j|+|a_j-1|)+||\psi-I_{2d}||+|\xi|\leq C\left[\left(\sum_j||f_j-g_j||_{p_j}\right)^2+||A^TJA||^2\right]
\end{equation}
and the orthogonality conditions of Proposition \ref{prop:only} hold for the functions
\begin{equation}
\tilde{F_j}(x)=a_jF_j(\psi(x)+v_j)e^{ix\cdot\xi+iA^TJA\tilde{v_j}\cdot x},
\end{equation}
where $\tilde{v_1}=v_2, \tilde{v_2}=v_1$, and $\tilde{v_3}=0$.
\end{lemma}

\begin{proof}
Begin by writing $h_j=g_j-F_j$ and $\tilde{h_j}=g_j-\tilde{F_j}$, where $\tilde{F_j}=a_jF_j(\psi(x)+v_j)e^{ix\cdot\xi+iA^TJA\tilde{v_j}\cdot x}$ and $\psi, v_j, a_j,\xi$ are to be determined. Letting $a_j=1+b_j$ and subbing in $f_j=g_j+h_j$,
\begin{align*}
\tilde{h_j}(x)&=a_jF_j(\psi(x)+v_j)e^{ix\cdot\xi+iA^TJA\tilde{v_j}\cdot x}-g_j(x)\\
&=(1+b_j)(g_j(\psi(x)+v_j)e^{ix\cdot\xi+iA^TJA\tilde{v_j}\cdot x}+h_j(\psi(x)+v_j)e^{ix\cdot\xi+iA^TJA\tilde{v_j}\cdot x})-g_j(x).
\end{align*}
Writing $\psi(x)=x+\phi(x)$ and taking the two terms involving $g_j$ from above,
\begin{align*}
g_j(\psi(x)&+v_j)e^{ix\cdot\xi+iA^TJA\tilde{v_j}\cdot x}-g_j(x)\\
&=g_j(x)[g_j^{-1}(x)g_j(x+v_j+\phi(x))e^{ix\cdot\xi+iA^TJA\tilde{v_j}\cdot x}-1]\\
&=g_j(x)(e^{-\pi p_j'[|x+v_j+\phi(x)|^2-|x|^2]}e^{ix\cdot\xi+iA^TJA\tilde{v_j}\cdot x}-1)\\
&=g_j(x)x\cdot[-2p_j'(\phi(x)+v_j)+i\xi+iA^TJA\tilde{v_j}]+O((||\phi||+|\bv|+|\xi|)^2),
\end{align*}
where $O((|\phi||+|\bv|+|\xi|)^2)$ represents the $L^{p_j}$ norm of the remainder term. Substituting back into the initial expression for $\tilde{h_j}$, one finds
\begin{multline}\label{eq:hj expression}
\tilde{h_j}(x)=a_jh_j(\psi(x)+v_j)e^{ix\cdot\xi+iA^TJA\tilde{v_j}\cdot x}\\+g_j(x)x\cdot[-2p_j'(\phi(x)+v_j)+i\xi+iA^TJA\tilde{v_j}]+O((||\phi||+|\bv|+|\mathbf{b}|+|\xi|)^2).
\end{multline}
In computing $\langle \tilde{h_j},P_\alpha^{(\tau_j)}g_j^{p_j-1}\rangle$, we begin with the main term from \eqref{eq:hj expression}.
\begin{align*}
\langle a_jh_j(\psi(x)&+v_j)e^{ix\cdot\xi+iA^TJA\tilde{v_j}\cdot x},P_\alpha^{(\tau_j)}g_j^{p_j-1}\rangle\\
&=\langle h_j(\psi(x)+v_j),P_\alpha^{(\tau_j)}g_j^{p_j-1}\rangle +O((|\mathbf{b}|+|\xi|+|\bv|)||h_j||_{p_j})\\
&=|\det(\psi)|^{-1}\int h_j(y)P_\alpha^{(\tau_j)}(\psi^{-1}(y-v_j)g_j^{p_j-1}(\psi^{-1}(y-v_j)dy\\
&\hspace{.5in}+O((|\mathbf{b}|+|\xi|+|\bv|)||h_j||_{p_j})\\
&=\langle h_j,P_\alpha^{(\tau_j)}g_j^{p_j-1}\rangle +O((|\mathbf{b}|+|\xi|+||\phi||+|\bv|)||h_j||_{p_j}).
\end{align*}
Considering the full expression from \eqref{eq:hj expression},
\begin{multline}
\langle \tilde{h_j},P_\alpha^{(\tau_j)}g_j^{p_j-1}\rangle=\langle h_j,P_\alpha^{(\tau_j)}g_j^{p_j-1}\rangle\\
+\langle g_j(x)x\cdot[b_j-2p_j'(\phi(x)+v_j)-i\xi-iA^TJA\tilde{v_j}],P_\alpha^{(\tau_j)}g_j^{p_j-1}\rangle\\
+O((||\phi||+|\bv|+|\mathbf{b}|+|\xi|)^2+(||\phi||+|\bv|+|\mathbf{b}|+|\xi|)||h_j||_{p_j}).
\end{multline}
In order to complete the proof via the Implicit Function Theorem, it suffices to show that the map
\begin{equation}\label{eq:invertible map}
(\mathbf{b},\bv,\xi,\phi)\mapsto \langle g_j(x)[b_j-v_j\cdot x-i(\xi+A^TJA\tilde{v_j})\cdot x-2p_j'x\cdot\phi(x)],P_\alpha^{(\tau_j)}g_j^{p_j-1}\rangle
\end{equation}
with $(j,\alpha)$ ranging over the indices specified in Proposition \ref{prop:only} and taking the real or imaginary part as specified is invertible.

Since $\{x\cdot\phi(x):\phi\text{ is a symmetric real }(2d)\times(2d)\text{ matrix}\}$ is precisely the set of symmetric, real, homogeneous, quadratic polynomials on $\R^{2d}$, the map $\phi\mapsto(\langle x\cdot\phi(x)g_3(x),P_\alpha^{(\tau_3)}g_3^{p_3-1}\rangle:|\alpha|=2)$ is invertible. These inner products vanish when $|\alpha|=0,1$.

The contribution from the mapping $(\bv,\xi)$ with the constraint $v_1+v_2+v_3=0$ to $\langle g_j(x)[v_j\cdot x-i(\xi+A^TJA\tilde{v_j})\cdot x],P_\alpha^{(\tau_j)}g_j^{p_j-1}\rangle$ ranging over the indices of Proposition \ref{prop:only} and taking the real and imaginary parts is also invertible. These products vanish when $|\alpha|=0,2$.

Lastly, the contribution from $\langle g_j(x)b_j,P_\alpha^{(\tau_j)}g_j^{p_j-1}\rangle$ indexed over $j=1,2,3$ is in one-to-one correspondence with $\mathbf{b}$ and these inner products vanish when $|\alpha|=1,2$. Thus, the maps described in \eqref{eq:invertible map} is invertible.
\end{proof}

\section{Putting it All Together}

\begin{proof}[Proof of Theorem \ref{theorem:main}]

Let $(h_1,h_2,h_3,B)$ be a 4-tuple with $h_j\in L^{p_j}$ and $B$ an arbitrary $(2d)\times(2d)$ matrix such that $\dist_\bp(\scriptO_{TC}(\bh,B),(\bg,0))$ is sufficiently small.
By the Balancing Lemma, there exists an element $(F_1,F_2,F_3,A)$ of the orbit of $(\bh,B)$ which satisfies the orthogonality conditions of Proposition \ref{prop:only}. Let $f_j=F_j-g_j$. Since
\begin{equation}
\dist_\bp(\scriptO_{TC}(\bh,B),(\bg,0))^2\leq \max_j||f_j||_{p_j}^2+||A^TJA||^2,
\end{equation}
it suffices to prove that
\begin{equation}
\frac{\T_A(\bg+\bf)}{\prod_j||g_j+f_j||_{p_j}}\leq A_\bp^{2d}-c\left[\max_j||f_j||_{p_j}^2+||A^TJA||^2\right].
\end{equation}
By Proposition \ref{prop:only},
\begin{equation}
\frac{\T_0(\bg+\bf)}{\prod_j||g_j+f_j||_{p_j}}\leq A_\bp^{2d}-c||\bf||_\bp^2-\tilde{c}\sum_j||f_{j,\flat}||_{p_j}^{p_j}.
\end{equation}
Thus, it suffices to show that
\begin{equation}
\frac{(\T_A-\T_0)(\bg+\bf)}{\prod_j||g_j+f_j||_{p_j}}\leq -c||A^TJA||^2+O((||\bf||_\bp+||A^TJA||)^3).
\end{equation}

We may ignore the product of norms in the denominator by appropriate modification of the constant $c$. Expanding $(\T_A-\T_0)(\bg+\bf)$ through the multilinearity of $\T_A-\T_0$, one obtains four types of terms. By Lemma \ref{lemma:onesigma} and Lemma \ref{lemma:twosigma},
\begin{multline*}
(\T_A-\T_0)(g_1,g_2,g_3)=\iint g_1(x)g_2(y)g_3(x+y)(e^{i\sigma(Ax,Ay)}-1)dxdy\hfill\\
\hfill=\iint g_1(x)g_2(y)g_3(x+y)(i\sigma(Ax,Ay)-\frac{1}{2}\sigma(Ax,Ay)^2+O(\sigma(Ax,Ay)^3))dxdy\\
\hfill=-c||A^TJA||^2+O(||A^TJA||^3).
\end{multline*}

By similar application of Lemma \ref{lemma:onesigma},
\begin{align*}
(\T_A-\T_0)(f_1,g_2,g_3)&=\iint f_1(x)g_2(y)g_3(x+y)(i\sigma(Ax,Ay)+O(\sigma(Ax,Ay)^2))dxdy\\
&\leq 0+||A^TJA||^2\int f_1(x)x^2\left[\int y^2g_2(y)g_3(x+y)dy\right]dx\\
&=O(||f_1||_{p_1}||A^TJA||^2)
\end{align*}
and likewise for all other terms involving one $f_j$ and two $g_j$'s.

The $(\T_A-\T_0)(f_1,f_2,f_3)$ term is negligible by Lemma \ref{lemma:trivial}, so only the terms with two $f_j$'s and one $g_j$ remain. Lemma \ref{lemma:3rd order} only addresses the situation where the $f_j$ are replaced with $f_{j,\sharp}$. However, Proposition \ref{prop:only} provides a $-\tilde{c}\sum_j||f_{j,\flat}||_{p_j}^{p_j}$ term which may be used here. Expanding further and applying Lemma \ref{lemma:3rd order} and \eqref{eq:trivial bound} gives
\begin{multline}
|(\T_A-\T_0)(f_1,f_2,g_3)|-\tilde{c}\sum_j||f_{j,\flat}||_{p_j}^{p_j}\\ \leq o(||\bf||_\bp^2+||A^TJA||^2)+ O(||f_{1,\sharp}||_{p_1}||f_{2,\flat}||_{p_2}+||f_{2,\sharp}||_{p_2}||f_{1,\flat}||_{p_1})-\tilde{c}\sum_j||f_{j,\flat}||_{p_j}^{p_j}.
\end{multline}
If $\sum_j||f_{j,\flat}||_{p_j}^{p_j}$ is small relative to $||\bf||_\bp^2$, then the above is negligible, as each $||f_{j,\flat}||_{p_j}$ is small. (Specifically, one may split into cases where $||f_{j,\flat}||_{p_j}\geq||f_j||_{p_j}^{(4-p_j)/2}$ for at least one $j$ or none of the $j$.) However, if $\sum_j||f_{j,\flat}||_{p_j}^{p_j}$ is large relative to $||\bf||_\bp^2$, then the last term dominates (as $p_j<2$), and the above is still negligible.

This holds for the other terms involving two $f_j$'s and one $g_j$, thus completing the proof of the main theorem.


\end{proof}

\end{document}